\documentclass[11pt, a4paper]{amsart}
\usepackage{amssymb, array, amsmath, amscd, pdfpages, enumerate, amsthm, setspace, mathtools}

\usepackage[utf8]{inputenc}
\usepackage{amssymb}
\usepackage{bm}
\usepackage{graphicx} 
\usepackage{overpic}
\usepackage[colorlinks=true,allcolors=magenta]{hyperref}

\newcommand{\yref}{y_{\mbox{\scriptsize\textit{ref}}}}
\newcommand{\wdist}{w_{\mbox{\scriptsize\textit{dist}}}}

\newcommand{\wdistk}[1][k]{w^{#1}_{\mbox{\scriptsize\textit{d}}}}

\newcommand{\CL}{C_\Lambda}

\newcommand{\Kz}{K_0}

\newcommand{\Kadd}{Q}
\newcommand{\KaddL}{Q_\Lambda}
\newcommand{\Yadd}{V}

\newcommand{\mc}[1]{\mathcal{#1}}

\newcommand{\citel}[2]{\cite[#2]{#1}}

\newtheorem{theorem}{Theorem}[section]

\newtheorem{lemma}[theorem]{Lemma}

\newtheorem{assumption}[theorem]{Assumption}

\theoremstyle{definition}
\newtheorem{definition}[theorem]{Definition}
\newtheorem{remark}[theorem]{Remark}
\newcommand{\pmat}[1]{\begin{bmatrix}#1 \end{bmatrix}}
\newcommand{\pmatsmall}[1]{\begin{bsmallmatrix}#1\end{bsmallmatrix}}

\DeclareMathOperator{\diag}{diag}

\newcommand*{\C}{{\mathbb{C}}}     
\newcommand*{\R}{{\mathbb{R}}}

\newcommand*{\N}{{\mathbb{N}}}

\newcommand*{\Lin}{{\mathcal{L}}}   
\newcommand*{\Dom}{D}

\newcommand*{\norm}[1]{\lVert#1\rVert}
\newcommand*{\set} [1]{\{#1\}}
\newcommand*{\setm}[2]{\{\,#1\mid#2\,\}}   
\newcommand*{\iprod}[2]{\langle#1,#2\rangle}

\newcommand*{\Lp}[1][p]{L^{#1}}
\newcommand*{\Xloc}[1]{#1_{\text{loc}}}
\newcommand*{\Lploc}[1][p]{\Xloc{L^{#1}}}
  
\newcommand{\G}{\mathcal{G}}

\newcommand{\Abs}[2][default]{\ifthenelse{\equal{#1}{default}}{\left\lvert#2\right\rvert}{\ldelim{#1}{\lvert}#2\rdelim{#1}{\rvert}}}
\newcommand{\Norm}[2][default]{\ifthenelse{\equal{#1}{default}}{\left\lVert#2\right\rVert}{\ldelim{#1}{\lVert}#2\rdelim{#1}{\rVert}}}

\newcommand*{\Iprod}[3][default]{\ifthenelse{\equal{#1}{default}}{\left\langle#2,#3\right\rangle}{\ldelim{#1}{\langle}#2,#3\rdelim{#1}{\rangle}}}
\newcommand*{\Dualpair}[3][default]{\ifthenelse{\equal{#1}{default}}{\left\langle#2,#3\right\rangle}{\ldelim{#1}{\langle}#2,#3\rdelim{#1}{\rangle}}}

\newcommand*{\List}[2][1]{\set{#1,\ldots,#2}}

\newcommand{\eq}[1]{\begin{align*}#1\end{align*}}
\newcommand{\eqn}[1]{\begin{align}#1\end{align}}

\newcommand{\ga}{\alpha}

\newcommand{\gl}{\lambda}
\newcommand{\gw}{\omega}

\newcommand{\inv}{^{-1}}

\newcommand*{\ddb}[2][1]{\ifthenelse{\equal{#1}{1}}{\frac{d}{d#2}}{\frac{d^{#1}}{d#2^{#1}}}}
\newcommand*{\pd}[3][1]{\ifthenelse{\equal{#1}{1}}{\frac{\partial{#2}}{\partial{#3}}}{\frac{\partial^{#1}{#2}}{\partial#3^{#1}}}}

\newtheorem*{RORP}{The Robust Output Regulation Problem}

\begin{document}

  \title[On Robust Regulation of PDEs]{On Robust Regulation of PDEs: from Abstract Methods to PDE Controllers}

\thanks{L. Paunonen was supported by the Academy of Finland grant 349002.
The research of J.-P. Humaloja was funded by the Jenny and Antti Wihuri Foundation and the Väisälä Fund. 
%was supported by the
 }

\thispagestyle{plain}

\author[L. Paunonen]{Lassi Paunonen}
\address[L. Paunonen]{Mathematics, Faculty of Information Technology and Communication Sciences, Tampere University, PO.\ Box 692, 33101 Tampere, Finland}
\email{lassi.paunonen@tuni.fi}

\author[J.-P. Humaloja]{Jukka-Pekka Humaloja}
\address[J.-P. Humaloja]{University of Alberta, 9211-116 St, Edmonton, AB T6G 1H9, Canada}
\email{jphumaloja@ualberta.ca}

%\thanks{This work was supported by }

\begin{abstract}
In this paper we study robust output tracking and disturbance rejection of linear partial differential equation (PDE) models. We focus on demonstrating how the abstract internal model based controller design methods developed for ``regular linear systems'' can be utilised in controller design for concrete PDE systems. We show that when implemented for PDE systems, the abstract control design methods lead in a natural way to controllers with ``PDE parts''. Moreover, we formulate the controller construction in a way which utilises minimal knowledge of the abstract system representation and is instead solely based on natural properties of the original PDE. We also discuss computation and approximation of the controller parameters, and illustrate the results with an example on control design for a boundary controlled diffusion equation.
\end{abstract}


\maketitle

\section{Introduction}
\label{sec:intro}

 Robust output regulation has been studied actively in the literature for controlled linear partial differential equations as well as for distributed parameter systems. 
In this control problem the aim is to achieve asymptotic convergence of the system's output to a predefined reference signal despite a class of external disturbance signals and uncertainties in the parameters of the system.
The primary motivation for studying the control problem for infinite-dimensional linear systems is that this abstract framework facilitates the study of \emph{classes} of linear PDE models and makes it possible to introduce general controller design methods which are applicable to a range of different types of PDEs. 
This way the abstract approach unifies and avoids repetition in the parts of the controller design which are independent of the type of the PDE model under consideration.
The output regulation problem adapts extremely well to the abstract infinite-dimensional setting because the associated controller design approaches 
have 
a lot of
 structure which is either independent or depends only in a very particular way on the considered system (e.g., through transfer function values or locations of transmission zeros).

When such an abstract controller construction method is applied in the control of a concrete PDE model,
the resulting controller is typically either a finite-dimensional ODE model (such as in~\cite{LogTow97,HamPoh00,RebWei03,PauPha20}), or alternatively an abstract linear system (in~\cite{Imm07a,HamPoh10,Pau16a,Pau17b,XuDub17a,VanBri22arxiv}).
In the latter case the
 natural expectation is that the controller is ``of similar type'' as the original system, namely, a PDE model.
The abstract controllers do indeed possess this intuitive property and this structure is easy to observe in the case of PDEs with distributed inputs and outputs. However, 
this relationship between the system and controller may become less obvious in the case of PDEs with boundary control and observation, where the abstract framework has a higher level of generality due to the unboundedness of the input and output operators.
Moreover, 
some of the controller construction algorithms require a certain level of technical knowledge on the abstract framework, and this can make the design methods tedious to implement for those researchers who are not already familiar with the corresponding abstract theory.

In this paper we demonstrate how a selected controller design method for abstract infinite-dimensional systems produces PDE controllers when applied in the control of PDEs with boundary control and observation. 
Moreover, we show that the controller design method can be presented
 in a way which requires minimal knowledge of the ``abstract framework'' and where the parameter choices are completely based on the original PDE system (in particular, stabilizability via feedback and output injections, and computation of selected transfer function values).
Our results are applicable for 
a wide range of boundary controlled PDEs in 1D (such as reaction--convection--diffusion equations, damped wave and beam equations, and coupled PDE-PDE and PDE-ODE systems), as well as $n$D convection--diffusion equations.

The observer-based robust controller~\citel{Pau16a}{Sec.~VI} studied in this paper consists of an ODE part (the \emph{internal model} of the reference and disturbance signals) and a modified copy of the controlled system which is used as a Luenberger-type observer in the stabilization of the closed-loop system.
As our main result we show that when applied in PDE control, the infinite-dimensional part of our controller is always a state of a PDE system which is of similar type as the original system.
We achieve this by rewriting the abstract controller in a new way and by analysing the detailed properties of the controller state. 
In this paper we allow the controlled system to be a general regular linear system, but for simplicity limit our attention to the situation where this model can be stabilized with state feedback and output injection with bounded operators. Using the results in~\cite{Pau17b}, our approach also generalises to the situation where the stabilization of the system requires boundary feedback or boundary output injection, but the controller form becomes somewhat more complicated.
Our approach can also be applied to other abstract controller structures (e.g., those for ``non-robust'' output regulation in~\cite{XuDub17a}) to design PDE-type controllers.

 \textbf{Notation.} For Hilbert spaces $X$ and $Y$ we denote the space of bounded linear operators $A:X \to Y$ by $\Lin(X,Y)$.
The resolvent operator of
 $A: \Dom(A)\subset X\to X$ is defined as $R(\gl,A)=(\gl I-A)\inv$ for $\gl\in \C$ in the resolvent set $\rho(A)$, 
and the adjoint of $A$ is denoted by $A^\ast: \Dom(A^\ast)\subset X\to X$.
The inner product on $X$ is denoted by $\iprod{\cdot}{\cdot}_X$.
We denote the $\Lambda$-extension~\citel{TucWei14}{Def.~5.1} of an operator $C$ by $\CL$.

\section{Preliminaries}
\subsection{The Robust Output Regulation Problem}

Throughout the paper consider controlled PDE systems with an input $u(t)\in U=\C^m$, a measured output $y(t)\in Y=\C^p$, and an additional disturbance input $\wdist(t)\in U_d=\C^{n_d}$.
The main objective in \emph{output regulation} is to design a dynamic error feedback controller 
so that 
 \emph{the output $y(t)$ converges asymptotically to a given reference signal $\yref(t)$ despite the external disturbance signals $\wdist(t)$.}
The considered reference and disturbance signals are of the form%
\begin{subequations}
  \label{eq:yrefwdist}
  \eqn{
  \label{eq:yrefwdistyr}
    \hspace{-1ex}  \yref(t) &=  \sum_{k=0}^q a_k \cos(\gw_k t+\theta_k) \\
  \label{eq:yrefwdistwd}
  \hspace{-1ex} \wdist(t) &=  \sum_{k=0}^q b_k \cos(\gw_k t + \varphi_k)
  }
\end{subequations}
where the frequencies $0=\gw_0<\gw_1<\ldots<\gw_q$ are known and the amplitudes $\set{a_k}_{k=1}^q\subset Y$ and $\set{b_k}_{k=1}^q\subset U_d$ and phases $\set{\theta_k}_{k=0}^q,\set{\varphi_k}_{k=0}^q\subset [0,2\pi)$ can be unknown.

Our main control problem, the ``Robust Output Regulation Problem''~\cite{RebWei03,HamPoh10} is defined in detail in the following.

\begin{RORP}
Construct a dynamic error feedback controller 
so that the following hold.
\begin{itemize}
\item[(a)] The closed-loop system consisting of the system and the controller is exponentially stable when $\wdist(t)\equiv 0$ and $\yref(t)\equiv 0$.
\item[(b)] There exists $\ga>0$ such that for all initial states of the system and the controller and for all 
$\set{a_k}_{k=1}^q$, $\set{b_k}_{k=1}^q$, $\set{\theta_k}_{k=0}^q$, and $\set{\varphi_k}_{k=0}^q$ in~\eqref{eq:yrefwdist}
 the output $y(t)$ satisfies
\eq{
\int_0^\infty e^{2\ga t}\norm{y(t)-\yref(t)}^2dt<\infty.
}
\item[(c)] If the parameters of the system are perturbed in such a way that the exponential closed-loop stability is preserved, then 
 \textup{(b)} still holds for some modified $\tilde{\ga}>0$.
\end{itemize}
\end{RORP}

\subsection{Assumptions on the PDE System}
\label{sec:plant}

As our main assumption we suppose that
the controlled PDE system can be expressed as a \emph{regular linear system}~\cite{Wei94,TucWei14}. 
Even though the regular linear system representation of the system is required in the proofs of our main results, our goal is to present the controller design and the controller structure in a way which is largely independent of this abstract formulation.
Instead it is mainly \emph{sufficient to know that such a representation exists}.
That being said, 
we assume the PDE has an abstract representation%
\begin{subequations}
\label{eq:plant}
\eqn{
\label{eq:plantstate}
\dot{x}(t)&=Ax(t)+Bu(t) + B_d\wdist(t), \qquad x(0)=x_0\\
y(t)&=\CL x(t)+Du(t) + D_d \wdist(t).
}
\end{subequations}
We assume $(A,[B,B_d],C,[D,D_d])$ is a regular linear system~\citel{TucWei14}{Sec.~5} on a Hilbert space $X$ with input space $U\times U_d=\C^m \times \C^{n_d}$ and output space $Y=\C^p$.
In particular, $A: \Dom(A)\subset X\to X$ generates a strongly continuous semigroup $T(t)$ on $X$. 
Our assumption also implies that for any $[u,\wdist]^T\in\Lploc[2](0,\infty;U\times U_D)$ and $x_0\in X$ the state $x(t)$ of the system is the unique \emph{mild solution} of~\eqref{eq:plantstate} (defined in~\citel{TucWei09book}{Def.~4.1.1}) 
 given by~\citel{TucWei09book}{Prop.~4.2.5}
\eq{
x(t) = T(t)x_0 + \int_0^t T(t-s)\left[ Bu(s)+B_d\wdist(s) \right]ds.
}
On the other hand, by~\citel{TucWei09book}{Rem.~4.2.6} the state 
also satisfies
\eqn{
\label{eq:plantweaksol}
\begin{aligned}
\hspace{-.3cm}\iprod{x(t)-x_0}{\phi}_X
= \int_0^t \Bigl[ &\iprod{x(s)}{A^\ast \phi}_{X} + \iprod{u(s)}{B^\ast \phi}_U \\
&\quad + \iprod{\wdist(s)}{B_d^\ast \phi}_{U_d} \Bigr]ds
\end{aligned}
}
for all $t> 0$ and $\phi\in \Dom(A^\ast)$.
It is important to note that it is precisely the identity~\eqref{eq:plantweaksol} which connects the state $x(t)$ of the system~\eqref{eq:plant} to the \emph{weak solution} of the original PDE system.
This relationship is illustrated in concrete examples in~\citel{TucWei09book}{Rem.~10.2.2,~10.2.4 and Sec.~10.7,~10.8}.

We make the following assumptions on the stabilizability and transmission zeros of the controlled PDE system. 

\begin{assumption}
\label{ass:StabAss}
Assume that there exists $\Kz \in \Lin(X,U)$ such that the state feedback $u(t)=\Kz x(t)$ stabilizes system~\eqref{eq:plant} exponentially. In addition, assume that there exists $L\in \Lin(Y,X)$ such that the output injection $Ly(t)$ stabilizes system~\eqref{eq:plant} exponentially.
\end{assumption}

The fact that $\Kz $ and $L$ are bounded operators in Assumption~\ref{ass:StabAss} means that we only consider systems which are stabilizable using distributed feedback and output injection.

\begin{remark}
In terms of regular linear systems Assumption~\ref{ass:StabAss} means that 
$\Kz \in \Lin(X,U)$ and $L\in \Lin(Y,X)$ are such that 
 the semigroups generated by 
 $A+LC: \Dom(A)\subset X\to X$ and $A+B\Kz : \Dom(A+B\Kz )\subset X\to X$ with domain $\Dom(A+B\Kz )=\setm{x\in X}{Ax+B\Kz x\in X}$ are exponentially stable.
\end{remark}

The following condition 
on transmission zeros 
 is necessary for the solvability of the robust output regulation problem.
\begin{assumption}
\label{ass:TZass}
The numbers of inputs and outputs 
of~\eqref{eq:plant} 
satisfy $m\geq p$ and~\eqref{eq:plant} does not have transmission zeros at $\set{\pm i\gw_k}_{k=0}^q$.
\end{assumption}

If we denote the transfer function of the system (from the input $u(t)$ to the output $y(t)$) by $P(\gl)$, then for any $i\gw_k\in\rho(A)$ the condition in Assumption~\ref{ass:TZass} requires that $P(\pm i\gw_k)$ has full row rank.
More generally, if $i\gw_k\in i\R$ the condition requires that the transfer function
 of the system stabilized with state feedback has full row rank \mbox{at $i\gw_k$.}

\section{Controller Design}

In this section we construct an error feedback controller which solves the robust output regulation problem. 
Our main result in Theorem~\ref{thm:RORPcontroller} shows that the controller state has a part which is the weak solution of a PDE of the same form as the original system.

The controller design is based on the construction of the parameters $(G_1,G_2,L,K)$ in Definition~\ref{def:ContrPar} below. 
%The parameters $G_1$ and $G_2$ have explicit formulas and $L$ and the part $\Kz$ of $K_2$ are chosen as in Assumption~\ref{ass:StabAss} based on the stabilizability properties of the original PDE system.
The construction uses the matrices $B_1^k\in \Lin(U,Y\times Y)$ and the operators $H_K^k\in \Lin(X,Y\times Y)$ defined by
\eq{
B_1^k &= 
\frac{1}{2} \pmat{P_K(i\gw_k) + P_K(-i\gw_k)\\iP_K(i\gw_k) -iP_K(-i\gw_k)}
,\quad
H_K^k = 
\frac{1}{2} \pmat{P_{KI}(i\gw_k) + P_{KI}(-i\gw_k)\\iP_{KI}(i\gw_k) -iP_{KI}(-i\gw_k)},
}
where
 $P_K(\gl)=(\CL+D\Kz)R(\gl,A+B\Kz)B+D$ 
and
 $P_{KI}(\gl)=(\CL+D\Kz)R(\gl,A+B\Kz)$. 
%The values $P_K(\pm i\gw_k)$ and $P_{KI}(\pm i\gw_k)$  can be computed based on the original PDE, and the details are presented in Section~\ref{sec:ContrParComputation}.

\begin{definition}[Controller Parameters]
\label{def:ContrPar}
Define $Z_0=\C^{p(2q +1)}$, 
\eq{
  G_1
= \diag(0_p, \gw_1\Omega_p,\ldots, \gw_q\Omega_p)\in \R^{p(2q+1)\times p(2q+1)},
}
with $\Omega_p = \pmatsmall{0_{p}& I_p\\-I_p&0_p}$,
where $0_p,I_p\in \R^{p\times p}$ are the 
 zero and identity matrices, and
\eq{
  G_2 = \bigl[I_p,I_p,0_p,I_p,0_p,\ldots,I_p,0_p\bigr]^T \in \R^{p(2q +1)\times p}.
}
Let $L\in \Lin(Y,X)$ 
and $\Kz \in \Lin(U,X)$ be as in Assumption~\ref{ass:StabAss}.
Define 
\eq{
B_1 = 
 \pmat{ P_K(0)\\B_1^1\\\vdots\\
B_1^q} 
\quad \mbox{and} \quad
H_K = 
 \pmat{ P_{KI}(0)\\ H_K^1\\ \vdots\\ H_K^q}.
}
The pair $(G_1,B_1)$ is controllable due to Assumption~\ref{ass:TZass} and  $K_1\in \Lin(Z_0,U)$ can be chosen so that $G_1+B_1K_1$ is Hurwitz. 
Finally, define $K_2 = \Kz +K_1H_K$.
\end{definition}

Theorem~\ref{thm:RORPcontroller} below presents a controller solving the robust output regulation problem
based on the parameters constructed in Definition~\ref{def:ContrPar}.
The theorem shows that 
 the controller 
 consists of an ODE system with state $z_1(t)$ (the ``internal model'') 
and an ``observer-part'' 
which is a copy of the system
 with input $u(t)$, output $\hat{y}(t)$, and an additional input with input operator $L$. 
In view of the discussion in Section~\ref{sec:plant} the result also shows that $\hat{x}(t)$ is a weak solution of a PDE of the same form as the original PDE system (with the additional input through the operator $L$ and with zero disturbance).
The controller can therefore be rewritten as a coupled PDE-ODE system, and this is illustrated further in the example considered in Section~\ref{sec:examples}.

\begin{theorem}
\label{thm:RORPcontroller}
Let $G_1$, $G_2$, $L$, and $K$ be as in Definition~\textup{\ref{def:ContrPar}} and let $e(t)=y(t)-\yref(t)$.
The robust output regulation problem is solved by the dynamic error feedback controller 
\begin{subequations}
\label{eq:maincontroller}
\eqn{
\hspace{-.4cm}\dot{z}_1(t)&=G_1 z_1(t) + G_2e(t), 
\hspace{2.2cm} 
z_1(0)\in Z\\
\dot{\hat{x}}(t)&=A\hat{x}(t)+Bu(t) + L(\hat{y}(t)-e(t)), ~ \hat{x}(0)\in X\\
\hat{y}(t) &= \CL \hat{x}(t) + Du(t)\\
u(t)&= K_1 z_1(t)+K_2\hat{x}(t).
}
\end{subequations}
With this controller the closed-loop system
(consisting of~\eqref{eq:plant} and~\eqref{eq:maincontroller})
 has a unique mild state $x_e(t)=[x(t),z_1(t),\hat{x}(t)]^T$, $u(\cdot)\in \Lploc[2](0,\infty;Y)$ $e(\cdot),\hat{y}(\cdot)\in\Lploc[2](0,\infty;Y)$ and $\hat{x}(\cdot)\in C([0,\infty);X)$ satisfies
\eq{
\iprod{\hat{x}(t)-\hat{x}(0)}{\phi}_X
= \int_0^t \Bigl[ &\iprod{\hat{x}(s)}{A^\ast \phi}_{X} + \iprod{u(s)}{B^\ast \phi}_U\\
& \quad + \iprod{L(\hat{y}(s)-e(s))}{ \phi}_{X} \Bigr]ds
}
for all $t\geq 0$ and $\phi\in \Dom(A^\ast)$.
\end{theorem}

\begin{remark}
%In the controller
 Definition~\ref{def:ContrPar} shows that
$G_1$ and $G_2$ have explicit formulas and that $L$ and the part $\Kz$ of $K_2$ are chosen as in Assumption~\ref{ass:StabAss} based on the stabilizability properties of the original PDE system.
Finally, the values $P_K(\pm i\gw_k)$ and $P_{KI}(\pm i\gw_k)$ in $B_1$ and $H_K$ can be computed based on the original PDE, 
%and the details are presented in 
as shown in Section~\ref{sec:ContrParComputation}.%
\end{remark}

The controller in Theorem~\ref{thm:RORPcontroller} is based on an abstract controller introduced in~\cite{HamPoh10,Pau16a} with general structure
\begin{subequations}
\label{eq:controller}
\eqn{
\label{eq:controllerstate}
\dot{z}(t)&= \G_1 z(t) + \G_2 e(t), \qquad z(0)=z_0\in Z\\
u(t)&= Kz(t)
}
\end{subequations}
with $e(t)=y(t)-\yref(t)$
on a Hilbert space $Z$. Here $\G_1$ generates a strongly continuous semigroup on $Z$ and $\G_2\in \Lin(Y,Z)$ and $K\in \Lin(Z,U)$.
For the proof of Theorem~\ref{thm:RORPcontroller}
we define the closed-loop system 
consisting of the system~\eqref{eq:plant} and the controller~\eqref{eq:controller}.
This closed-loop system has state $x_e(t)=[x(t),z(t)]^T$ on $X_e=X\times Z$ and is of the form
\begin{subequations}
\label{eq:CLsys}
\eqn{
\dot{x}_e(t) &= A_ex_e(t) + B_e w_e(t), \qquad x_e(0)=x_{e0}\\
e(t)&= C_e x_e(t) + D_e w_e(t)
}
\end{subequations}
where $w_e(t) = [\wdist(t),\yref(t)]^T$, $x_{e0}=[x_0,z_0]^T$, 
\eq{
A_e = \pmat{A&BK\\ \G_2\CL&\G_1+\G_2DK}, \qquad B_e=\pmat{B_d&0\\0&-\G_2},
}
with domain $\Dom(A_e)=\setm{[x,z]^T\in \Dom(\CL)\times \Dom(\G_1)}{Ax+BKz\in X}$, and
$C_e=[\CL,DK]$ and $D_e=[0,-I]$. The closed-loop system $(A_e,B_e,C_e,D_e)$ is a regular linear system~\citel{Pau16a}{Thm.~3}.
 The following additional properties of the closed-loop system are used in the proof of Theorem~\ref{thm:RORPcontroller}.

\begin{lemma}
\label{lem:CLstructure}
Let $x_e(t)=[x(t),z(t)]^T$ be the mild state of~\eqref{eq:CLsys}.
Then $z(t)$ is the mild solution of the differential equation~\eqref{eq:controllerstate}.
Moreover, if $\Yadd$ is a Hilbert space and $\Kadd: \Dom(\G_1)\subset Z\to \Yadd$  is an admissible output operator for the semigroup generated by $\G_1$, then $z(t)\in \Dom(\KaddL)$ for a.e. $t\geq 0$ and $\KaddL z(\cdot)\in \Lploc[2](0,\infty;\Yadd)$.
\end{lemma}

\begin{proof}
Consider an open loop system~\citel{Pau17b}{Thm.~2.3}
\eq{
\biggl(\pmat{A&0\\0&\G_1},\pmat{B&B_d&0\\0&0&\G_2},\pmat{C&0\\0&K},\pmat{D&0&0\\0&0&0}\biggr)
}
with input $[u(t),\wdist(t),u_c(t)]^T$ and output $[y(t),y_c(t)]^T$.
It is easy to see that this is a regular linear system on $X_e=X\times Z$. 
The closed-loop system~\eqref{eq:CLsys} is obtained from the open loop system by applying the admissible output feedback
\eq{
\pmat{u(t)\\\wdist(t)\\u_c(t)}=\pmat{0&I\\0&0\\I&0}\pmat{y(t)\\y_c(t)}+\pmat{0\\\wdist(t)\\-\yref(t)},
}
subsequently ignoring the first input and the second output, and finally adding the feedthrough term $D_ew_e(t)$. This feedback structure together with~\citel{Wei94}{Thm.~6.1 and Eq.~(6.1)} imply that $z(t)$ is indeed the mild solution of~\eqref{eq:controllerstate}.

To prove the second claim we can note that since $\G_2$ is a bounded operator,  $(\G_1,\G_2,\Kadd,0)$ is a regular linear system. Since $z(t)$ is the mild solution of~\eqref{eq:controllerstate}  and since the regulation error satisfies $e(\cdot)\in\Lploc[2](0,\infty;Y)$ (as the output of the regular closed-loop system), we have $z(t)\in \Dom(\KaddL)$ for a.e. $t\geq 0$ and $\KaddL z(\cdot)\in \Lploc[2](0,\infty;\Yadd)$ by~\citel{Wei94}{Thm.~5.5}.
\end{proof}

\begin{proof}[Proof of Theorem~\textup{\ref{thm:RORPcontroller}}]
Definition~\ref{def:ContrPar} and~\citel{Pau16a}{Thm.~15}\footnote{The result~\citel{Pau16a}{Thm.~15} assumes
 that the system has an equal number of inputs and outputs, i.e., $m=p$. However, the result and its proof remain valid for $m\geq p$ under Assumption~\ref{ass:TZass} 
since $G_1+B_1K_1$ is Hurwitz by the choice of $K_1$.
} imply that the
robust output regulation problem is solved by an
 abstract controller of the form~\eqref{eq:controller} on $Z=Z_0\times X$ with state $z(t)=[z_1(t),\hat{x}(t)]^T$ and 
with parameters
\begin{equation*}
\begin{aligned}
\mathcal{G}_1&=\begin{bmatrix}
G_1&0\\
(B+LD)K_1&A+L\CL+(B+LD)K_2
\end{bmatrix}\\
\Dom(\mc{G}_1)&=\setm{\pmatsmall{z_1\\x}\in Z_0\times \Dom(\CL)}{Ax+BK\pmatsmall{z_1\\x}\in X}\\
\mathcal{G}_2&=\begin{bmatrix}
G_2\\
-L
\end{bmatrix}, \quad
K=\pmat{K_1,~K_2}, 
\quad K_2=\Kz +K_1H_K.
\end{aligned} 
\end{equation*}
The closed-loop system has a well-defined mild state $x_e(t)=[x(t),z_1(t),\hat{x}(t)]^T$. 
Thus it remains to show that 
$[z_1(t),\hat{x}(t)]^T$ is the mild state of the controller~\eqref{eq:maincontroller} and that 
$u(\cdot)$, $e(\cdot)$, $\hat{y}(\cdot)$, and $\hat{x}(\cdot)$ have the claimed properties.

Define 
$\Kadd = \pmatsmall{K_1&K_2\\0&C}$ with $\Dom(\Kadd)=Z_0\times \Dom(A)$.
We have 
\eq{
\G_1 = \pmat{G_1&0\\0&A+L\CL} + \pmat{0\\B+LD}\pmat{I&0} \pmat{K_1&K_2\\0&\CL},
}
where 
\eq{
\biggl(\pmat{G_1&0\\0&A+LC}, \pmat{0\\B+LD},\pmat{K_1&K_2\\0&C}, \pmat{0\\0}\biggr)
}
is a regular linear system. 
We therefore have from~\citel{TucWei14}{Thm.~5.17} 
that
$\Kadd$ is 
an admissible output operator for the semigroup generated by $\G_1$ and its
 $\Lambda$-extension is given by
 $\KaddL=\pmatsmall{K_1&K_2\\0&\CL}$ with $\Dom(\KaddL)=Z_0\times \Dom(\CL)$.
  Lemma~\ref{lem:CLstructure} thus implies
that 
$z(t)\in \Dom(\KaddL) = Z_0\times \Dom(\CL)$ for a.e. $t\geq 0$ and $\KaddL z(\cdot)\in \Lploc[2](0,\infty;Y)$. 
But since $\hat{y}(t) =\CL \hat{x}(t) + Du(t)=[D,I]\KaddL z(t)  $, this immediately implies 
$\hat{y}(\cdot) \in \Lploc[2](0,\infty;Y)$.
Moreover, the regularity of the closed-loop system implies
 $e(\cdot)\in \Lploc[2](0,\infty;Y)$, and thus also the output $u(t)$ of~\eqref{eq:controller} satisfies
$u(\cdot)\in\Lploc[2](0,\infty;U)$.

By Lemma~\ref{lem:CLstructure} the function $z(t) = [z_1(t),\hat{x}(t)]^T$ is the mild solution of~\eqref{eq:controllerstate}.
Since $Z_0$ is finite-dimensional, the triangular structure of $\G_1$
 implies that $z_1(t)$ is the (strong) solution of 
$\dot{z}_1(t)=G_1 z_1(t) + G_2e(t)$.
Moreover, the structure of $\G_1$ 
 and $u(t)=K_1z_1(t) + K_2 \hat{x}(t)$
also imply that $\hat{x}(t)$ is the mild solution of the differential equation
\eq{
\dot{\hat{x}}(t)
 &= (B+LD)K_1z_1(t) 
 + (A+L\CL +(B+LD)K_2) \hat{x}(t) - L e(t)\\
 &= A \hat{x}(t)+ (B+LD)(K_1z_1(t) + K_2 \hat{x}(t)) 
 +L(\CL  \hat{x}(t) -  e(t))\\
 &= A \hat{x}(t)+ Bu(t) +L(  \hat{y}(t)   -  e(t)).
}
By~\citel{TucWei09book}{Rem.~4.2.6} this means that $\hat{x}(\cdot)$ is continuous and that it satisfies the integral equation in the claim.
\end{proof}

\section{Computing the Controller Parameters}
\label{sec:ContrParComputation}

In this section we describe how
the values $P_K(\pm i\gw_k)$ and $P_{KI}(\pm i\gw_k)$ used
 in the controller construction
 can be computed based on the original PDE system.

\subsection{The General Transfer Function Approach}
\label{sec:ContrParGeneral}

The definitions 
\eq{
P_K(\gl) &= (\CL+D\Kz)R(\gl,A+B\Kz)B+D\\
 P_{KI}(\gl) &= (\CL+D\Kz)R(\gl,A+B\Kz)
}
 imply that
 $[P_K(\gl),P_{KI}(\gl)]$ is the transfer function of the regular linear system $(A+B\Kz ,[B,I],C+D\Kz ,[D,0])$. This is precisely
 the system $(A,[B,I],C,[D,0])$ under state feedback $[u(t),\psi(t)]^T=[\Kz x(t)+\tilde{u}(t),\psi(t)]^T$ (see Fig.~\ref{fig:SysAddInput}).

\begin{figure}[h!]
\begin{center}
%\begin{overpic}[width=0.57\linewidth]{CDC22-Additional-Input.pdf}
%\put(49,6){\mbox{\normalsize$\Kz$}}
%\put(38,31){\mbox{\normalsize PDE System}}
%\put(1,30.5){\mbox{\normalsize$\tilde{u}(t)$}}
%\put(21,23){\mbox{\footnotesize$+$}}
%%\put(19,30.5){\mbox{\normalsize$u(t)$}}
%\put(1,42){\mbox{\normalsize$\psi(t)$}}
%\put(85,42){\mbox{\normalsize$y(t)$}}
%\put(78,30.5){\mbox{\normalsize$x(t)$}}
%\end{overpic}
\includegraphics[width=.57\linewidth]{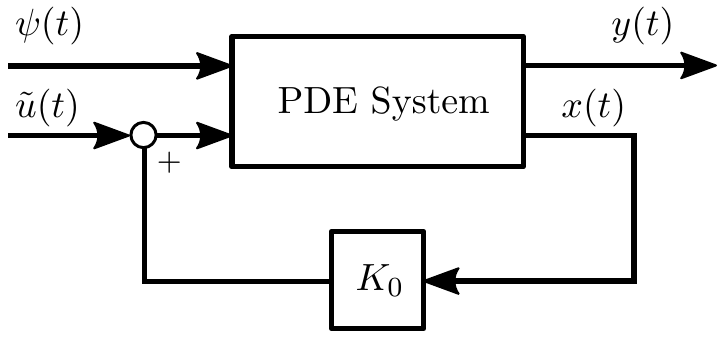}
\caption{The system structure for computing $P_K(\pm i\gw_k)$ and $P_{KI}(\pm i\gw_k)$.}
\label{fig:SysAddInput}
\end{center}
\end{figure}

This feedback structure and
 the fundamental properties of transfer functions imply that the values of $P_K(i\gw) $ and $P_{KI}(i\gw)$ for $\gw\in\set{ \pm \gw_k}_{k=0}^q$ 
can be computed from the original PDE system 
  in the following way (cf.~\cite{Zwa04,CurMor09}): 
\begin{quotation}
\emph{
Add a new distributed input $\psi(t)$
 to the PDE system
 corresponding to the input operator $I\in \Lin(X)$.
Let $u_0\in U$, $\psi_0\in X$, and $\gw\in \R$.
If $x_0$ is the (unique) initial data of the PDE system such that the weak solution corresponding to the input $(u(t),\psi(t))=(e^{i\gw t}\Kz x_0+e^{i\gw t}u_0,e^{i\gw t}\psi_0)$ has the form $x(t)=e^{i\gw t}x_0$, then the corresponding output has the form $y(t)=e^{i\gw t}y_0$ where $y_0=P_K(i\gw)u_0+P_{KI}(i\gw)\psi_0$.
}
\end{quotation}

As noted in~\citel{CurMor09}{Sec.~1.1}, the above approach (after elimination of the common factors $e^{i\gw t}$) leads to the same static differential equation for $x_0\in X$
 as taking the formal Laplace transform of the PDE system with an additional distributed input $\psi(t)$, under state feedback $[u(t),\psi(t)]^T=[\Kz x(t)+\tilde{u}(t),\psi(t)]^T$, and with zero initial condition.
This same property can also be observed in the abstract system~\eqref{eq:plant}: 
It is easy to use~\citel{TucWei09book}{Rem.~4.2.6} to show that $x_0\in X$ has the above properties if and only if 
\eq{
i\gw\iprod{x_0}{\phi}_X
= \iprod{x_0}{A^\ast \phi}_{X}  \hspace{-.03cm} + \hspace{-.03cm}  \iprod{u_0 \hspace{-.03cm} + \hspace{-.03cm} \Kz x_0}{B^\ast \phi}_U  \hspace{-.03cm} + \hspace{-.03cm}  \iprod{\psi_0}{ \phi}_X 
}
for all $\phi\in \Dom(A^\ast)$. This equation coincides (in a weak sense) with the formal Laplace transform of the corresponding differential equation (with zero initial condition).

\begin{remark}
\label{rem:ContrParamBCS}
The differential equation for computing $P_K(i\gw) $ and $P_{KI}(i\gw)$ for $\gw\in\set{ \pm \gw_k}_{k=0}^q$ has a particularly concrete representation if the original PDE can be expressed as an \emph{abstract Boundary Control System}
\eq{
\dot{x}(t)&= \mc{A}x(t)+B_d^0\wdistk[0](t), \qquad x(0)=x_0 \\
\mc{B}x(t)&= u(t)+\wdistk[1](t),
 \qquad \mc{B}_dx(t)=\wdistk[2](t)\\
y(t)&=\mc{C}x(t),
}
where $\mc{A}: \Dom(\mc{A})\subset X\to X$ is a differential operator and $\mc{B}\in \Lin(\Dom(\mc{A}),U)$ and $\mc{B}_d\in \Lin(\Dom(\mc{A}),U_d)$ are boundary trace operators (see~\cite{Sal87a,CheMor03,MalSta06} for details). In this situation, the above approach (and elimination of the common factors $e^{i\gw t}$) shows that 
if $u_0\in U$, $\psi_0\in X$ and $\gw\in\R$ and if $x_0\in \Dom(\mc{A})$ is the solution of the boundary value problem
\begin{subequations}
\label{eq:PKBVP}
\eqn{
(i\gw - \mc{A})x_0&=\psi_0, \\
\label{eq:PKBVPBC}
\mc{B}x_0&= \Kz x_0+ u_0, \quad \mc{B}_dx_0=0,
}
\end{subequations}
then $y_0 = \mc{C}x_0=P_K(i\gw)u_0+P_{KI}(i\gw)\psi_0$. 
In particular, $x_0\in X$ satisfies the boundary conditions of the static differential equation~\eqref{eq:PKBVP}.
\end{remark}

\begin{remark}
As shown in~\citel{Pau16a}{Thm.~15}, the operator $H_K\in \Lin(X,Z_0)$ is the solution of the Sylvester equation
 \eq{
G_1H_K=H_K(A+B\Kz )+G_2(\CL+D\Kz )
}
defined on $\Dom(A+B\Kz )$. However, we emphasize that 
 $H_K$ has an explicit formula based on $P_{KI}(\pm i\gw_k)$, and solving this operator equation is not required.
On the other hand, in certain situations such as for parabolic equations the operator $H_K$ can be approximated reliably with a solution of the Sylvester equation projected onto a finite-dimensional space.
\end{remark}

\subsection{Reduction to Simpler Systems}
\label{sec:ContrParPertForm}

In the case where $i\gw\in \rho(A)$ the values $P_K(i\gw)$ and $P_{KI}(i\gw)$ can be computed based on solutions of simpler differential equations. 
Standard properties of transfer functions show that for any $\gl\in\rho(A)\cap \rho(A+B\Kz )$ 
we have
\eq{
P_K(\gl)u_0&=  P(\gl)(I-G_K(\gl))\inv u_0\\
P_{KI}(\gl)\psi_0&=  C R(\gl,A) \psi_0 +P_K(\gl) \Kz  R(\gl,A) \psi_0
}
 where $G_K(\gl):=\Kz R(\gl,A) B$ is the transfer function of the system $(A,B,\Kz ,0)$.
The system
\eqn{
\label{eq:ContrParExtSys}
\left( A,[B,I],\pmat{C\\\Kz },\pmat{D&0\\0&0} \right),
}
is the original PDE system with an additional distributed input $\psi(t)$ corresponding to the input operator $I\in \Lin(X)$ and with an additional output with operator $\Kz $.
The transfer function of~\eqref{eq:ContrParExtSys} is given by
\eq{
\pmat{P(\gl)&CR(\gl,A)\\G_K(\gl)&K_0R(\gl,A)},
}
and thus its values at $\gl=\pm i\gw_k$ contain the necessary information 
for computing $P_K(\pm i\gw_k)$ and $P_{KI}(\pm i\gw_k)$.
This transfer function of~\eqref{eq:ContrParExtSys} can be computed using the same approach as in Section~\ref{sec:ContrParGeneral} (but without 
the state feedback).

\subsection{Numerical Approximations}
\label{sec:ContrParNumerical}

Due to the internal model structure of the controller the output tracking and disturbance rejection are achieved whenever the parameter $K=[K_1,K_2]$ of the controller is such that the closed-loop system is exponentially stable.  Replacing $K$ with $\tilde{K}$ in the closed-loop system leads to
\eq{
\tilde{A}_e 
= \pmat{A&B\tilde{K}\\ \G_2\CL&\G_1+\G_2D\tilde{K}}
= A_e + \pmat{B\\\G_2D}\pmat{0,\;\tilde{K}-K}.
}
Since the nominal values $K_1$ and $K_2=\Kz +K_1H_K$ are guaranteed to stabilize the closed-loop semigroup $T_e(t)$ generated by $A_e$ and since $\pmatsmall{B\\\G_2D}$ is an admissible input operator for $T_e(t)$, the closed-loop system is stable for any $\tilde{K}$ for which $\norm{\tilde{K}-K}$ is sufficiently small.
Because of this, we can replace $K_1$ and $H_K$ in the controller with any approximations $\tilde{K}_1$  and $\tilde{H}_K$ for which $\norm{\tilde{K}_1-K_1}$ and $\norm{\tilde{H}_K-H_K}$ are sufficiently small. This immediately implies that it is sufficient to compute the
values $[P_K(\pm i\gw_k),P_{KI}(\pm i\gw_k)]$ with finite numerical accuracy, e.g., using software
for solving the boundary value problems in 
Sections~\ref{sec:ContrParGeneral} and~\ref{sec:ContrParPertForm}.

 Moreover, since $\dim Y=p$, the operators $P_{KI}(\pm i\gw_k)$ are compact and can be approximated 
 with finite-rank operators.
For any orthonormal basis $\set{\psi_n}_{n\in\N}$ of $X$ we can define 
\eq{
H_K^N = \sum_{n=1}^N \iprod{\cdot}{\psi_n}P_{KI}(\pm i\gw_k)\psi_n,
}
and the approximation error $\norm{H_K^N-H_K}$  can be made arbitrarily small 
with a sufficiently large $N\in\N$.
This shows that it is sufficient to compute 
 $P_{KI}(\pm i\gw_k)\psi_0$ for $\psi_0=\psi_n$ for a finite number of basis functions $n\in \List{N}$. In the method presented in Sections~\ref{sec:ContrParGeneral} and~\ref{sec:ContrParPertForm} this means that 
\emph{only a finite number of boundary value problems with
$\psi_0=\psi_n$, $n\in \List{N}$, need to be solved} (and each of these can be solved numerically).

\section{Controller Design for Heat Equations}
\label{sec:examples}

As a concrete model we consider a heat equation
    \eq{
      x_t(\xi,t) &= \Delta x(\xi,t) + B_d^0(\xi)\wdistk[0](t), \quad x(\xi,0)=x_0(\xi) \\
      \pd{x}{n}(\xi,t)\vert_{\partial \Omega} &= b(\xi)u(t)+B_d^1(\xi)\wdistk[1](t)\\
      y(t) &= \int_{\partial \Omega}x(\xi,t)c(\xi)d\xi, \qquad 
    }
on a one or two-dimensional spatial domain $\Omega\subset \R^n$. We assume that $\Omega=(a,b)$ if $n=1$ and that $\Omega$ is bounded and convex with piecewise $C^2$-boundary if $n=2$.
The system has scalar-valued input $u(t)\in \R$ and output $y(t)\in \R$ acting on the boundary with $b,c\in\Lp[2](\partial \Omega;\R)$, $b\neq 0$ and $c\neq 0$.
We assume $\wdistk[0](t)\in\R^{n_{d0}}$, $\wdistk[1](t)\in\R^{n_{d1}}$, $B_d^0(\cdot)\in \Lp[2](\Omega;\R^{1\times n_{d0}})$, and $B_d^1(\cdot)\in \Lp[2](\partial\Omega;\R^{1\times n_{d1}})$.
The PDE defines a regular linear system on $X=\Lp[2](\Omega)$~\citel{ByrGil02}{Thm.~2}, and it is unstable due to eigenvalue at $0\in\C$.

Theorem~\ref{thm:RORPcontroller} shows that if the parameters $G_1$, $G_2$, $K_1$, $K_2$, and $L$ are as in Definition~\ref{def:ContrPar}, then the robust output regulation problem 
 is solved by the controller 
\eq{
\dot{z}_1(t)&= G_1z_1(t) + G_2 (y(t)-\yref(t))\qquad z_1(0)\in Z_0\\
      \hat{x}_t(\xi,t) &= \Delta \hat{x}(\xi,t) + L(\xi) \bigl(\int_{\partial \Omega}\hspace{-1.5ex}\hat{x}(\xi,t)c(\xi)d\xi-y(t)+\yref(t)\bigr) \\
\MoveEqLeft[2.3]      \pd{\hat{x}}{n}(\xi,t)\vert_{\partial \Omega} = b(\xi)(K_1 z_1(t)+K_2 \hat{x}(\cdot,t)),
~ \hat{x}(\xi,0)= \hat{x}_0(\xi) \\
u(t)&=K_1 z_1(t)+K_2 \hat{x}(\cdot,t).
}
In particular,
$\hat{x}(\cdot,\cdot)$ is the weak solution of the PDE in the controller equations.
To construct the controller parameters, we can first choose $G_1$ and $G_2$ as in Definition~\ref{def:ContrPar} corresponding to the output space $Y=\C$ and the frequencies $0=\gw_0<\gw_1<\ldots<\gw_q$ in the considered reference and disturbance signals.
The stabilization of the system 
 can be achieved using LQR design~\cite{BanIto97} or (if $n=1$) explicit choices of 
the bounded $\Kz$ and $L$.
The results in Section~\ref{sec:ContrParGeneral} show that for $\gw=\pm \gw_k\in\R$ the values $P_K(i\gw)$ and $P_{KI}(i\gw)$
 can be computed by solving the boundary value problem%
\eq{
      i\gw x_0(\xi) &= \Delta x_0(\xi) + \psi_n(\xi) \\
      \pd{x_0}{n}(\xi)\vert_{\partial \Omega} &= b(\xi)(u_0 + \Kz x_0(\cdot)), \quad
      y_0 = \int_{\partial \Omega} \hspace{-1ex} x_0(\xi)c(\xi)d\xi.
    }
With the choices $u_0=1\in \C$ and $\psi_n=0\in \Lp[2](0,1)$ we then have $y_0=P_K(i\gw)$, and for $u_0=0$ and $\psi_n\in \Lp[2](0,1)$ we get $y_0=P_{KI}(i\gw)\psi_n$.
As in Section~\ref{sec:ContrParNumerical} the boundary value problem can be solved numerically 
 and it suffices to compute $y_0$
 for a finite number of $\psi_n$ from an orthonormal basis of $\Lp[2](\Omega)$.
In the 1D case the equations become ODEs,
 $\partial \Omega = \set{a,b}$, and $\Kz x_0 = -\int_a^b x_0(\xi)k_0(\xi)d\xi$ for some $k_0\in\Lp[2](a,b)$. 
Such boundary value problems can be solved easily and with very high precision using the free \textbf{Chebfun} MATLAB library~\cite{DriHal14book} (available at 
\href{https://www.chebfun.org/}{www.chebfun.org}).

MATLAB simulation codes for a 2D heat equation on a rectangle and a 1D heat equation
 (with spatially varying heat conductivity) 
are available at
 \href{https://github.com/lassipau/CDC22-Matlab-simulations/}{github.com/lassipau/CDC22-Matlab-simulations/}. The codes utilise 
the \textbf{RORPack} MATLAB library 
(\href{https://github.com/lassipau/rorpack-matlab/}{github.com/lassipau/rorpack-matlab/}) 
 and 
 \textbf{Chebfun} in the controller construction.
Fig.~\ref{fig:Simres} illustrates the 2D simulation example and results.

\begin{figure}[h!]
\begin{center}
\includegraphics[width=0.6\linewidth]{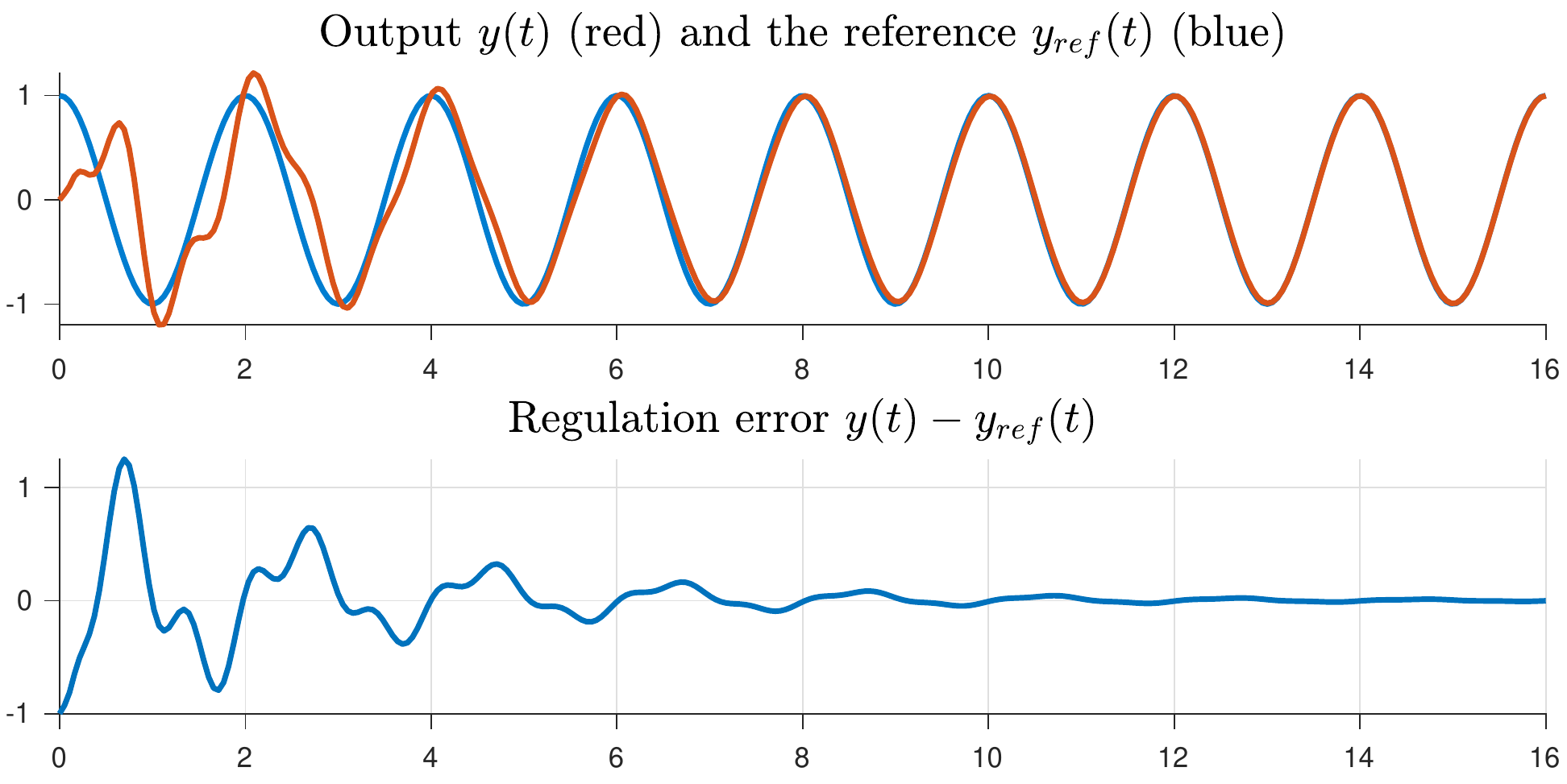}

\begin{minipage}{0.4\linewidth}
\vspace{.3cm}
\begin{center}
%\begin{overpic}[width=0.7\linewidth]{Heat2D-figure}
%\put(18,67){\rotatebox{-90}{\footnotesize disturbance}}
%\put(25,17){\mbox{\footnotesize input}}
%\put(38,80){\mbox{\footnotesize output}}
%\end{overpic}
\includegraphics[width=0.7\linewidth]{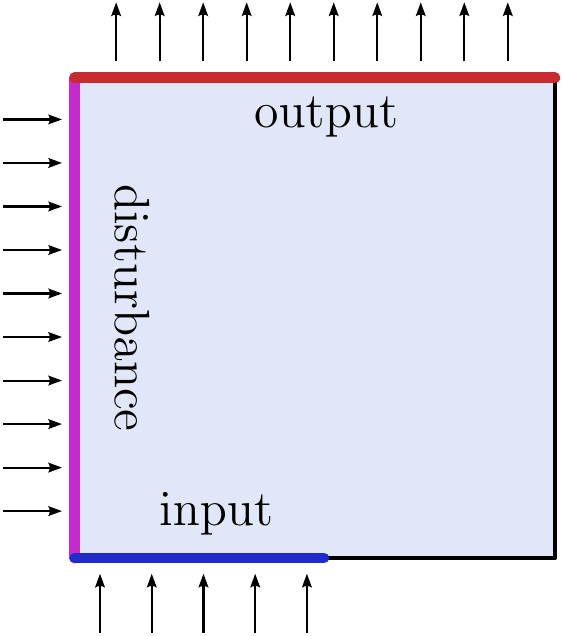}
\end{center}
\vspace{.2cm}
\end{minipage}
\hspace{.5cm}
\begin{minipage}{0.5\linewidth}
\vspace{.3cm}
\begin{center}
\includegraphics[width=0.8\linewidth]{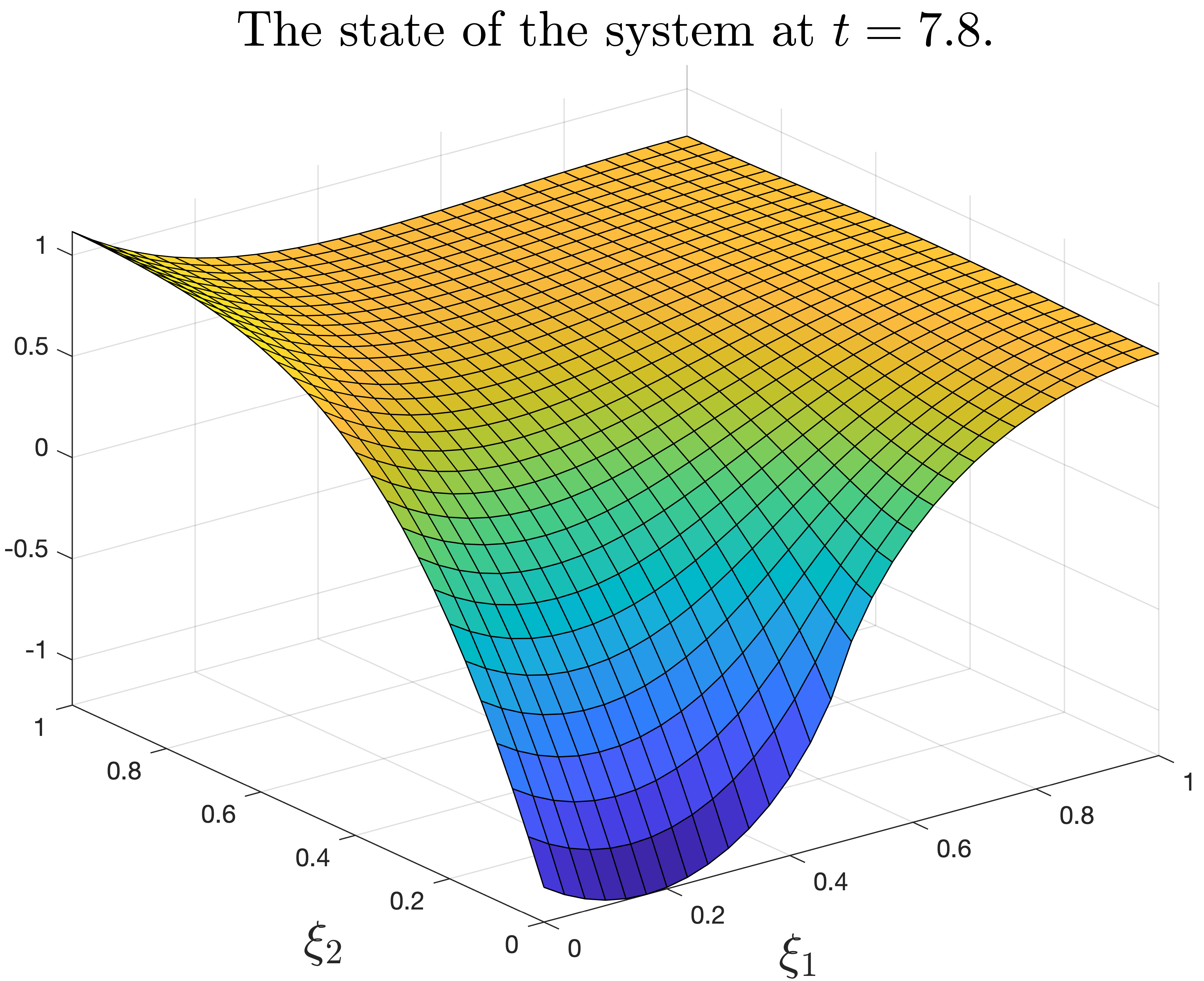}
\end{center}
\vspace{.2cm}
\end{minipage}
\caption{2D heat equation: controlled output (top), tracking error (middle), configuration (bottom left) and temperature profile $t=7.8s$ (bottom right).}
\label{fig:Simres}
\end{center}
\end{figure}

%\bibliographystyle{plain}
%\bibliography{../../../../reference}

\begin{thebibliography}{10}

\bibitem{BanIto97}
H.~Thomas Banks and Kazufumi Ito.
\newblock Approximation in {LQR} problems for infinite-dimensional systems with
  unbounded input operators.
\newblock {\em J. Math. Systems Estim. Control}, 7(1):1--34, 1997.

\bibitem{ByrGil02}
Christopher~I. Byrnes, David~S. Gilliam, Victor~I. Shubov, and George Weiss.
\newblock Regular linear systems governed by a boundary controlled heat
  equation.
\newblock {\em J. Dyn. Control Syst.}, 8(3):341--370, 2002.

\bibitem{CheMor03}
Ada Cheng and Kirsten Morris.
\newblock Well-posedness of boundary control systems.
\newblock {\em SIAM J. Control Optim.}, 42(4):1244--1265, 2003.

\bibitem{CurMor09}
Ruth Curtain and Kirsten Morris.
\newblock Transfer functions of distributed parameter systems: a tutorial.
\newblock {\em Automatica J. IFAC}, 45(5):1101--1116, 2009.

\bibitem{DriHal14book}
Tobin~A. Driscoll, Nicholas Hale, and Lloyd~N. Trefethen.
\newblock {\em Chebfun Guide}.
\newblock Pafnuty Publications, 2014.

\bibitem{HamPoh00}
Timo H{\"a}m{\"a}l{\"a}inen and Seppo Pohjolainen.
\newblock A finite-dimensional robust controller for systems in the
  {CD}-algebra.
\newblock {\em IEEE Trans. Automat. Control}, 45(3):421--431, 2000.

\bibitem{HamPoh10}
Timo H{\"a}m{\"a}l{\"a}inen and Seppo Pohjolainen.
\newblock Robust regulation of distributed parameter systems with
  infinite-dimensional exosystems.
\newblock {\em SIAM J. Control Optim.}, 48(8):4846--4873, 2010.

\bibitem{Imm07a}
Eero Immonen.
\newblock On the internal model structure for infinite-dimensional systems:
  {T}wo common controller types and repetitive control.
\newblock {\em SIAM J. Control Optim.}, 45(6):2065--2093, 2007.

\bibitem{LogTow97}
Hartmut Logemann and Stuart Townley.
\newblock Low-gain control of uncertain regular linear systems.
\newblock {\em SIAM J. Control Optim.}, 35(1):78--116, 1997.

\bibitem{MalSta06}
Jarmo Malinen and Olof~J. Staffans.
\newblock Conservative boundary control systems.
\newblock {\em J. Differential Equations}, 231(1):290--312, 2006.

\bibitem{Pau16a}
Lassi Paunonen.
\newblock Controller design for robust output regulation of regular linear
  systems.
\newblock {\em IEEE Trans. Automat. Control}, 61(10):2974--2986, 2016.

\bibitem{Pau17b}
Lassi Paunonen.
\newblock Robust controllers for regular linear systems with
  infinite-dimensional exosystems.
\newblock {\em SIAM J. Control Optim.}, 55(3):1567--1597, 2017.

\bibitem{PauPha20}
Lassi {Paunonen} and Duy {Phan}.
\newblock Reduced order controller design for robust output regulation of
  parabolic systems.
\newblock {\em IEEE Trans. Automat. Control}, 65(6):2480--2493, 2020.

\bibitem{RebWei03}
Richard Rebarber and George Weiss.
\newblock Internal model based tracking and disturbance rejection for stable
  well-posed systems.
\newblock {\em Automatica J. IFAC}, 39(9):1555--1569, 2003.

\bibitem{Sal87a}
Dietmar Salamon.
\newblock Infinite-dimensional linear systems with unbounded control and
  observation: {A} functional analytic approach.
\newblock {\em Trans. Amer. Math. Soc.}, 300(2):383--431, 1987.

\bibitem{TucWei09book}
Marius Tucsnak and George Weiss.
\newblock {\em Observation and Control for Operator Semigroups}.
\newblock Birkh\"auser Basel, 2009.

\bibitem{TucWei14}
Marius Tucsnak and George Weiss.
\newblock Well-posed systems---{T}he {LTI} case and beyond.
\newblock {\em Automatica J. IFAC}, 50(7):1757--1779, 2014.

\bibitem{VanBri22arxiv}
Nicolas {Vanspranghe} and Lucas {Brivadis}.
\newblock {Output regulation of infinite-dimensional nonlinear systems: a
  forwarding approach for contraction semigroups}.
\newblock {\em arXiv e-prints}, page arXiv:2201.10146, January 2022.

\bibitem{Wei94}
George Weiss.
\newblock Regular linear systems with feedback.
\newblock {\em Math. Control Signals Systems}, 7(1):23--57, 1994.

\bibitem{XuDub17a}
Xiaodong Xu and Stevan Dubljevic.
\newblock Output and error feedback regulator designs for linear
  infinite-dimensional systems.
\newblock {\em Automatica J. IFAC}, 83:170--178, 2017.

\bibitem{Zwa04}
Hans Zwart.
\newblock Transfer functions for infinite-dimensional systems.
\newblock {\em Systems Control Lett.}, 52(3--4):247--255, 2004.

\end{thebibliography}

\end{document}